\documentclass[english]{StyleFiles/llncs}
\usepackage[utf8]{inputenc}
\usepackage{amsmath, amssymb, bm, graphicx, mathdots, url}
\usepackage[hidelinks, colorlinks = true, allcolors=blue]{hyperref}

\DeclareMathOperator{\ILP}{ILP}
\DeclareMathOperator{\rank}{rank}

\newcommand{\Z}{\mathbb{Z}}
\newcommand{\N}{\mathbb{N}}

\newcommand{\bA}{\bm{A}}
\newcommand{\bx}{\bm{x}}
\newcommand{\bc}{\bm{c}}
\newcommand{\bb}{\bm{b}}

\newcommand{\bv}{\bm{v}}
\newcommand{\bu}{\bm{u}}
\newcommand{\bM}{\bm{M}}
\newcommand{\bT}{\bm{T}}
\newcommand{\bD}{\bm{D}}
\newcommand{\bR}{\bm{R}}
\newcommand{\indet}{x}
\newcommand{\intvar}{a}
\newcommand{\intvarfixed}{a}
\newcommand{\polyring}{\Z[\indet]}
\newcommand{\polyS}{S}
\newcommand{\polyM}{\bM}

\title{On Matrices over a Polynomial Ring with Restricted Subdeterminants}

\author{Marcel Celaya\inst{1}
	 \and Stefan Kuhlmann\inst{2} 
	 \and Robert Weismantel\inst{2}}
	 
\institute{
	School of Mathematics, Cardiff University, Wales, United Kingdom
	\and Department of Mathematics, 
	 ETH Z\"{u}rich, Switzerland
}

\begin{document}
	\maketitle
	\thispagestyle{plain}
	
	\noindent
	\textbf{Abstract.} This paper introduces a framework to study discrete optimization problems which are parametric in the following sense: their constraint matrices correspond to matrices over the ring $\polyring$ of polynomials in one variable.
	We investigate in particular matrices whose subdeterminants all lie in a fixed set $S\subseteq\polyring$. Such matrices, which we call \emph{totally $S$-modular matrices}, are closed with respect to taking submatrices, so it is natural to look at minimally non-totally $S$-modular matrices which we call \emph{forbidden minors for $S$}. Among other results, we prove that if $S$ is finite, then the set of all determinants attained by a forbidden minor for $S$ is also finite. Specializing to the integers, we subsequently obtain the following positive complexity results: the recognition problem for totally $\pm\{0,1,a,a+1,2a+1\}$-modular matrices
	with $a\in\Z\backslash\lbrace -3,-2,1,2\rbrace$ and the integer linear optimization problem for totally $\pm\lbrace 0,a,a+1,2a+1\rbrace$-modular matrices with $a\in\Z\backslash\lbrace -2,1\rbrace$ can be solved in polynomial time.

	\section{Introduction}
	
	For a matrix $\bM$, let $\Delta(\bM)$ denote the maximal absolute value of a subdeterminant of $\bM$. Since decades it is well-known that the number $\Delta(\bM)$ plays a crucial role in understanding complexity questions related to integer programming problems whose associated constraint matrix is $\bM$. Important examples include bounds on the diameter of polyhedra \cite{bonisummaeisenbranddiameterpoly14,dadushhahnle2016shadowsimplex,narayananshahsri2021spectraldiameter}, questions about the proximity between optimal LP solutions and optimal integral solutions \cite{alievhenkoertel2020knapsackprox,celayakuhlpaatweis2023proxandflatness}, bounds on the size of the minimal support of integer vectors in standard form programs \cite{alievAverkovLoeraOertel21,alievdeloera2018supportintegeroptimal,alievdeloesparelindio2017,eisenbrandshmonincaratheodorybounds06}, and running time functions for dynamic programming algorithms \cite{eisenweissteinitz18}. Largely unexplored however remain two fundamental algorithmic questions: Given a finite set $S\subseteq \N$ of allowed values for the subdeterminants in absolute value of the matrix $\bM$:
	\begin{enumerate} 
		\item What is the complexity of solving an ILP with associated constraint matrix $\bM$ in dependence of $\left| S\right|$ and $\max_{s\in S} s$? (Optimization problem)
		\item How can we efficiently verify whether a matrix $\bM$ has all its subdeterminants in absolute value in $S$? (Recognition problem)
	\end{enumerate}
	If $S = \lbrace 0,1\rbrace$, both questions have been answered. Indeed, a famous theorem of Hoffmann and Kruskal \cite{hoffkruskal1956tustatement} and a decomposition theorem of Seymour for totally unimodular matrices \cite{SEYMOUR1980305} allow us to tackle both questions. We refer to \cite{schrijvertheorylinint86} for further theory concerning totally unimodular matrices. If $S=\lbrace 0,1,2\rbrace$, there exists a polynomial time algorithm to solve Question 1 \cite{artweiszenbimodalgo2017}. In this generality, further optimization results are not known to us. There are other important results when one imposes additional restrictions on the constraint matrix. For instance, if one assumes that each row of the constraint matrix contains at most two non-zero entries, there is a polynomial time algorithm for the optimization problem \cite{Fiorini2022IntegerPW}. A randomized polynomial time algorithm for the related integer feasibility question can be derived if the constraint matrix is $\lbrace 0,\Delta\rbrace$-modular for $\Delta\leq 4$ \cite{naegele2023advancesstrictlydelta,naegelesanzencongruence2021}. There are also polynomial time algorithms for the optimization problem if the constraint matrix is $\lbrace a,b,c\rbrace$-modular \cite{glanzer2022notes} for $a,b,c \in \N$.  
	
	The main reason why there are only few general results on optimization and recognition problems is due to the lack of understanding integral matrices with bounded subdeterminants. 
	The purpose of this paper is to develop a framework to investigate the structure of those matrices. Our point of departure is to consider matrices with entries being elements in $\polyring$, the ring of polynomials with integral coefficients in one indeterminate $\indet$. Moreover, we specify a set of polynomials $\polyS\subseteq\polyring$ that corresponds to the allowed polynomials for the subdeterminants of those matrices. All other polynomials in $\polyring\backslash \polyS$ are forbidden. Let us make this precise. 
	\begin{definition}
		\label{def_tot_S_mod_forbidden_minor}
		Let $\polyS\subseteq \polyring$ be finite. Let $1\leq m,n$. A matrix $\polyM\in \polyring^{m\times n}$ is \emph{totally $\polyS$-modular} if every $k\times k$ subdeterminant of $\polyM$ is contained in $\polyS$ for $1\leq k\leq \min \lbrace m,n\rbrace$. Let $2\leq l$. The matrix $\polyM\in \polyring^{l\times l}$ is a \emph{forbidden minor for} $\polyS$ if every $(l-1)\times (l-1)$ submatrix is totally $\polyS$-modular but $\det\polyM\notin \polyS$. By $F(\polyS)$, we denote the set of all polynomials that arise as a determinant of some forbidden minor for $\polyS$.
	\end{definition}
	The case $\Z$ can be recovered by restricting $\polyS$ to consist of constant polynomials. 
	One advantage of operating in $\polyring$ is that we can extract a certain decomposition, Theorem \ref{thm_linear_form_decomposition}, and simplify arguments due to arithmetic properties of $\polyring$. This allows us to make progress towards Question 1 and 2 if we evaluate the polynomials in $\polyM$ and $\polyS$ at integers. The disadvantage of our approach is that there are finitely many values of $\intvar\in\Z$ for which our statements do not hold. For those values $\intvar$, there exist a polynomial in $F(\polyS)$ and a polynomial in $\polyS$ whose evaluations at $\intvar$ admit the same value. This implies that $F(S(\intvar))$ cannot be the set of all polynomials that arise as a determinant of some forbidden minor for the set $S(\intvar)$ of evaluations of all polynomials of $\polyS$ at $\intvar$; cf. Lemma \ref{lemma_existence_equal_sets}. 
	
	Let us remark that there is a related line of research. It involves understanding the matroids that admit a representation as a $\pm\lbrace 0,1,2,\ldots,\Delta\rbrace$-modular matrix for $\Delta\geq 2$ \cite{oxleywalsh2022bimodular,paatstallknecht2024forbiddenminors}. We emphasize that there are crucial differences between this approach and the totally $S$-modular direction proposed in this work. For instance, here the notion of a forbidden minor depends on the concrete representation of the matrix.
	
	We will later be able to make general statements about the structure of the set $F(\polyS)$. However, when it comes to Questions 1 and 2, we need further restrictions to derive general statements. 
	
	\subsection{The smallest non-trivial cases}
	Given the requirements that $\indet\in S$ and $S$ is as small as possible, it turns out that the first non-trivial cases are given by matrices with associated sets $\polyS = \pm\lbrace 0,\indet,\indet + 1,2\indet + 1\rbrace$ and $\polyS = \pm\lbrace 0,1,\indet,\indet+1,2\indet+1\rbrace$. Let us next argue why this is the case. 
	We assume that $0\in \polyS$. This assumption can be made without loss of generality because, if $0\notin \polyS$, we are quite restricted; see, for instance, \cite{ARTMANN2016635}.
	So let $\polyS=\pm\lbrace 0, \indet\rbrace$. 
	One can check that there is no invertible totally $\polyS$-modular matrix in dimension larger than one. 
	Hence, we may assume that $\polyS$ contains at least three different polynomials, i.e., $\polyS=\pm\lbrace 0,\indet,y\rbrace$ for $y\in\polyring$. The only ways of choosing $y$ such that the set of totally $\polyS$-modular matrices contains invertible instances in every dimension is either $y = 1$ or $y = \indet+1$. Let $S = \pm \lbrace 0,\indet,\indet + 1\rbrace$. 
	After removing all-zero rows and columns, one can show that totally $S$-modular matrices $\bM$ are submatrices of the following matrix up to a sign, row and column permutations, and duplicates:
	\begin{align*}	
		\begin{pmatrix}
			\indet+1& \dots& \dots& \indet + 1 \\ \vdots& & \iddots& \indet \\ \vdots& \iddots& \iddots &\vdots  \\ \indet+1& \indet&\dots &\indet
		\end{pmatrix}.
	\end{align*}
	In other words, the rows and columns of $\bM$ can be totally ordered, or, 
	equivalently, $\bM$ excludes the submatrix
	\begin{align}
		\label{matrix_forbidden_conflict}
		\begin{pmatrix}
			\indet+1 & \indet\\ \indet & \indet+1
		\end{pmatrix}.
	\end{align}
	This implies that the matrix (\ref{matrix_forbidden_conflict}) is a forbidden minor for $\polyS$ with determinant $2\indet+1$.
	We therefore obtain a purely combinatorial description which does not take into account the knowledge about subdeterminants but still completely characterizes totally $\polyS$-modular matrices.  
	One of the advantages of this observation is that, after minor preprocessing, the recognition problem becomes surprisingly straightforward in large dimensions: Up to row and column permutations and multiplying with minus one, one can simply check whether the matrix has entries in $\lbrace \indet,\indet + 1\rbrace$ and the forbidden submatrix in (\ref{matrix_forbidden_conflict}) does not appear. 
	This leads to the natural question of what happens when we allow the matrix (\ref{matrix_forbidden_conflict}) to be present, i.e., we add $2\indet+1$ to the set $\polyS$. 
	This gives us $\polyS = \pm\lbrace 0,\indet,\indet+1,2\indet+1\rbrace$. As a next step, one can add the choice $y = 1$ from above to our current set to obtain $\polyS = \pm \lbrace 0,1,\indet,\indet + 1, 2\indet + 1\rbrace$.
	
	\subsection{Statements of our results}
	The statements are concerned with matrices that can be decomposed into $\polyM = \bT + \indet\cdot\bu\cdot\bv^\top$ for integral $\bT,\bu$, and $\bv$. To derive complexity results, we need to understand the matrix $\bT$ for a given set $\polyS$. The following decomposition result applies. Let $\polyS(a)$ denote the set of evaluations of all polynomials of $\polyS$ at $a\in\Z$.
	
	\begin{theorem}\label{thm_linear_form_decomposition}
		Let $\polyS\subseteq \polyring$ be finite. Let $\polyM=\bT + \indet\cdot\bu\cdot\bv^\top$ where $\bT,\bu,$ and $\bv$ are integral. If $\polyM$ is totally $\polyS$-modular, then $\bT$ is totally $\polyS(0)$-modular.

	\end{theorem}
	For instance, let $K_0\subseteq\N$ be finite and  $\polyS=\pm\lbrace k\indet-1,k\indet,k\indet+1: k\in K_0\rbrace$. Then Theorem \ref{thm_linear_form_decomposition} states that $\bT$ is totally unimodular.

	Theorem \ref{thm_linear_form_decomposition} can be significantly generalized. The assumption that $\polyM = \bT + \indet\cdot\bu\cdot\bv^\top$ can be replaced by $\polyM = \bT+\indet\cdot \polyM_1 + \ldots + \indet^l\cdot\polyM_l $ where $\polyM_1,\ldots,\polyM_l$ are arbitrary integral matrices for $l\in \N$. We refrain from proving this in detail. The reason is that we only use Theorem \ref{thm_linear_form_decomposition}, and its evaluated form Corollary~\ref{cor_decomposition_evaluation}, to derive statements about optimization and recognition, but even then Theorem \ref{thm_linear_form_decomposition} itself is not enough.
	Further properties of the matrix $\bT$ are required to tackle the following constraint parametric optimization problem over $\Z$ which is given in terms of a full-column-rank constraint matrix $\bM(\intvar)\in \Z^{m\times n}$ where every entry of $\bM(\intvar)$ is parametrized by a specific value $a\in\Z$: 
	\begin{align*}
		\ILP(\bM(\intvar),\bb,\bc): \max \bc^\top\bx \ \text{s.t.} \ \bM(\intvar)\bx\leq\bb, \ \bx\in\Z^n
	\end{align*}
	where $\bb$ is integral. 
	The following two results can be derived with the tools developed in this paper.
	\begin{theorem}\label{thm_recognition}
		Let $\intvar\in \Z\backslash\lbrace -3,-2,1,2\rbrace$ and $S(\intvar)=\pm\lbrace 0,1,\intvar,\intvar+1,2\intvar+1\rbrace$. Given a matrix $\bM(\intvar)\in \Z^{m\times n}$, one can decide in polynomial time whether $\bM(\intvar)$ is totally $S(\intvar)$-modular.
	\end{theorem}
	\begin{theorem}\label{thm_opt_a_a_a+1}
	Let $\intvar\in \Z\backslash\lbrace -2,1\rbrace$ and $S(\intvar)=\pm\lbrace 0,\intvar,\intvar+1,2\intvar+1\rbrace$. Let $\bM(\intvar)\in \Z^{m\times n}$ have full column rank and be totally $S(\intvar)$-modular. Then one can solve $\ILP(\bM(\intvar),\bb,\bc)$ for integral $\bb$ and $\bc$ in polynomial time.
	\end{theorem}
	
	\section{Tools}\label{sec_tools}
	Throughout this paper we work with the lexicographical order of $\polyring$ which is defined by $s < t$ if and only if the leading coefficient of $t - s$ is positive. With respect to this ordering we further define the absolute value of $s\in \polyring$ to be
	\begin{align*}
		|s| = \begin{cases}
			s, & \text{if }s \geq 0,\\
			-s, & \text{otherwise} 
		\end{cases}.
	\end{align*}
	Also, we interchangeably pass from polynomials $s\in \polyring$ to polynomial functions, denoted by $s(a)$, when evaluating polynomials. This can be done since the polynomial function is uniquely determined by the polynomial over $\polyring$; see \cite[Chapter~4]{langalgebra}. 
	Let $I\subseteq \lbrack n\rbrack$ and $J \subseteq \lbrack n\rbrack$. We denote by $\bM_{\backslash I,\backslash J}$ the submatrix of $\bM$ without the rows indexed by $I$ and columns indexed by $J$. In what follows, the notation $\bA_{(k)}\subseteq\bM$ is shorthand for specifying that $\bA_{(k)}$ is a $k\times k$ submatrix of $\bM$. 
Given a submatrix $\bA$ of $\bM$ and $i,j\in\lbrack n\rbrack$, we denote by $\bA[i,j]$ the submatrix of $\bM$ that contains $\bA$ along with the extra row and column indexed by $i$ and $j$ respectively.

	We now show two results that can be viewed as a generalization of the well-known fact that every matrix that is not totally unimodular and has entries in $\pm\lbrace 0,1\rbrace$ contains a submatrix with determinant two in absolute value. For that purpose, we utilize a well-known determinant identity from linear algebra which is due to Sylvester \cite{sylvester1851relationminors} and commonly referred to as \emph{Sylvester's determinant identity}. By convention, the determinant of an empty matrix is 1.
	\begin{lemma}[Sylvester's determinant identity]
		\label{lemma_sylvester}
		Let $\bM\in \polyring^{n\times n}$ for $n\geq 2$ and let $\bA_{(k)}\subseteq \bM$ for $k=0,\ldots,n$ with $\bA_{(k)} = \bM_{\backslash I,\backslash J}$ for the ordered sets $I=\lbrace i_1,\ldots,i_{n-k}\rbrace$ and $J=\lbrace j_1,\ldots, j_{n-k}\rbrace$. Then we get
		\begin{align*}
			\det\bM\cdot(\det\bA_{(k)})^{n-1-k} = \det\begin{pmatrix}
				\det \bA_{(k)}[i_1,j_1]& \ldots & \det \bA_{(k)}[i_1,j_{n-k}] \\
				\vdots & \ddots & \vdots \\
				\det \bA_{(k)}\lbrack i_{n-k}, j_1\rbrack & \ldots & \det \bA_{(k)}\lbrack i_{n-k},j_{n-k}\rbrack
			\end{pmatrix}.
		\end{align*}
	\end{lemma}
	For the most part, we work with the special case when $k = n - 2$, which is also known as the \emph{Desnanont-Jacobi identity}. In this case, we get the equation
	\begin{align}\label{eq_desnanont_jacobi}
		\det\bM\cdot\det\bA_{(k)} = \det\begin{pmatrix}
			\det \bA_{(k)}[i_1,j_1]& \det \bA_{(k)}[i_1,j_2] \\
			\det \bA_{(k)}\lbrack i_2, j_1\rbrack & \det \bA_{(k)}\lbrack i_2,j_2\rbrack
		\end{pmatrix}
	\end{align}
	for $I=\lbrace i_1,i_2\rbrace$ and $J=\lbrace j_1,j_2\rbrace$. 
	This identity already implies our first bound:
	
	\begin{lemma}\label{lemma_finite_number_of_determinants}
		Let $\polyS\subseteq \polyring$ be finite. Then the set $F(\polyS)$ of determinants attained by the forbidden minors of $\polyS$ is finite and 
		\begin{align*}
			\max_{x\in F(\polyS)}|x|\leq 2\cdot \max_{s\in \polyS}s^2.
		\end{align*} 
	\end{lemma}
	\begin{proof}
		Select some forbidden minor $\polyM$ for $\polyS$ such that $\det\polyM\neq 0$. 
		There exists an invertible submatrix $\bA_{(n-2)}\subseteq \bM$. By the Desnanot-Jacobi identity (\ref{eq_desnanont_jacobi}) applied to $\bA_{(n-2)}$, we obtain that 
		\begin{align*}
			\det\polyM = \frac{1}{\det\bA_{(n-2)}}\cdot\left( s_1s_2-s_3s_4\right),
		\end{align*}
		where $s_i\in \polyS$ for $i=1,2,3,4$. Since $\det\bA_{(n-2)}\in S$, the right hand side only attains finitely many values as $\polyS$ is finite. Hence, there are only finitely many values possible for $\det\polyM$. 
		To obtain the inequality, we take absolute values on both sides, apply the triangle inequality, and observe that $\left|s_i\right|\leq \max_{s\in \polyS}\left|s\right|$. The claim follows then from $1\leq \left|\det\bA_{(n-2)}\right|$.\qed
	\end{proof}

	We immediately obtain that every forbidden minor for totally $\pm\lbrace 0, 1,\ldots,\Delta\rbrace$-modular matrices over $\Z$ admits a determinant bounded by $2\Delta^2$ in absolute value. The bound in Lemma \ref{lemma_finite_number_of_determinants} is tight: Let $\polyS\subseteq\polyring$ be finite such that $s\in \polyS$ implies $-s\in \polyS$ and $\lbrace 0\rbrace \neq \polyS$. Select $\tau = \max_{s\in S}\left|s\right|$. Then
	\begin{align*}
		\begin{pmatrix} \tau & \tau \\ -\tau & \tau\end{pmatrix}
	\end{align*}
	is a forbidden minor for $S$ and has determinant $2\tau^2$. One can strengthen the bound in Lemma \ref{lemma_finite_number_of_determinants} if the dimension is sufficiently large and $S\subseteq \Z$. 
	\begin{theorem}\label{thm_bound_minor}
		Let $S\subseteq \Z$ be finite. Let $\tau = \max_{s\in S}\left|s\right|$. Given a forbidden minor $\bM\in\Z^{n\times n}$ for $S$ of dimension $n\geq\left\lceil \log_2\tau + 1\right\rceil$, then
		\begin{align*}
			\left|\det\bM\right| \leq 2\cdot\lceil\log_2\tau + 1\rceil\cdot\tau.
		\end{align*}
	\end{theorem}
	\begin{proof}
		Suppose that the matrix $\bM$ is invertible, otherwise we are done. Let $\Delta_k = \max\left\lbrace \left|\det\bA_{(k)}\right| : \bA_{(k)}\subseteq\bM\right\rbrace$ for all $k=0,\ldots,n$ and $\kappa = \lceil \log_2\tau\rceil$. If $\kappa = 0$, we have $\tau = 1$ and, thus, $\Delta_{n-1}/\Delta_{n-2}\leq 1$ as $\Delta_{n-2}\geq 1$ and $\Delta_{n-1} \leq \tau = 1$. So the Desnanont-Jacobi identity (\ref{eq_desnanont_jacobi}) applied to some invertible $(n-2)\times (n-2)$ submatrix implies 
		\begin{align*}
			\left|\det\bM\right|\leq 2 \cdot \Delta_{n-1}/\Delta_{n-2}\cdot \Delta_{n-1}\leq 2\cdot\Delta_{n-1}\leq 2\cdot \tau.
		\end{align*}
		
		For the remainder of the proof, we suppose that $1\leq\kappa$. Observe that $\kappa + 1\leq n$. Assume that $\Delta_{n - j}/\Delta_{n - j - 1}> 2$ for all $j=1,\ldots,\kappa$. This yields 
		\begin{align*}
			\prod_{j=1}^{\kappa} \Delta_{n-j}/\Delta_{n-j-1}> 2^{\kappa}\geq\tau.
		\end{align*} 
		However, we also have
		\begin{align*}
			\prod_{j=1}^{\kappa} \Delta_{n-j}/\Delta_{n-j-1} = \Delta_{n-1}/\Delta_{n-\kappa-1}\leq \tau
		\end{align*}
		as $\Delta_{n-\kappa - 1}\geq 1$ and $\Delta_{n-1}\leq\tau$, which is a contradiction. So we know there exists an index $l^*\in\lbrack \kappa\rbrack$ such that $\Delta_{n - l^*}/\Delta_{n-l^*-1}\leq 2$, where we have $n - l^* - 1\geq 0 $ by construction. Let $\bA:=\bA_{(n - l^* - 1)}\subseteq\bM$ attain $\Delta_{n - l^* - 1}$. Then applying Sylvester's determinant identity, Lemma \ref{lemma_sylvester}, to $\bA$ gives us
		\begin{align*}
			\det\bM\cdot(\det\bA)^{l^*} = \det\underbrace{\begin{pmatrix}
				\det \bA[i_1,j_1]& \ldots & \det \bA[i_1,j_{l^*+1}] \\
				\vdots & \ddots & \vdots \\
				\det \bA\lbrack i_{l^*+1}, j_1\rbrack & \ldots & \det \bA\lbrack i_{l^*+1},j_{l^*+1}\rbrack
			\end{pmatrix}}_{=\bD}
		\end{align*}
		for suitable sets $I$ and $J$. 
		Dividing by $(\det\bA)^{l^*}$ and applying Laplace expansion to the first row on the right hand side yields
		\begin{align*}
			\det\bM = \frac{1}{(\det\bA)^{l^*}}\cdot \sum_{k = 1}^{l^*+1}(-1)^{k}\cdot\det\bA[i_1,j_k]\cdot\det\bD_{\backslash 1,\backslash k}.
		\end{align*}
		Observe that $\det\bD_{\backslash 1,\backslash k} = \det\bM_{\backslash i_1,\backslash j_k}\cdot(\det\bA)^{l^* - 1}$ by Lemma \ref{lemma_sylvester}. Hence,
		\begin{align*}
			\det\bM = \frac{1}{\det\bA}\cdot\sum_{k = 1}^{l^* + 1}(-1)^{k}\cdot\det\bA[i_1,j_k]\cdot\det\bM_{\backslash i_1,\backslash j_k}.
		\end{align*}
		Taking absolute values and using that $\left|\det\bA\right| = \Delta_{n - l^* - 1}$ gives 
		\begin{align*}
			\left|\det\bM\right| \leq (l^*+1) \cdot\Delta_{n-l^*}/\Delta_{n-l^*-1}\cdot\Delta_{n-1}\leq (\kappa + 1) \cdot 2\cdot \tau.
		\end{align*}
		The claim follows from $\kappa = \lceil \log_2\tau\rceil$.\qed
	\end{proof}
	
	Let $\polyS\subseteq \polyring$ be finite. If $\polyM$ is totally $\polyS$-modular over $\polyring$, then $\bM(\intvar)$ is totally $S(\intvar)$-modular for every $\intvar\in\Z$. 
	This raises the question of whether totally $S(\intvarfixed)$-modular matrices over $\Z$ for a fixed value $\intvarfixed\in\Z$ are also totally $\polyS$-modular over $\polyring$, i.e., totally $S(\intvar)$-modular for all $\intvar\in\Z$. This is in general not true: Let $\polyS = \pm\lbrace 0,\indet,\indet+1,2\indet+1\rbrace\subseteq \polyring$ and $\intvarfixed = 1$. We define the matrix
	\begin{align*}
		\polyM = \begin{pmatrix} \indet & \indet+1 & \indet & \indet\\ 
		\indet+1 & \indet+1 & \indet+1 & \indet \\ 
		\indet & \indet+1 & \indet+1 & \indet+1 \\ 
		\indet & \indet & \indet+1 & \indet 
		\end{pmatrix}\in\polyring^{4\times 4}.
	\end{align*}
	One can check that $\bM(1)$ is totally $S(1)$-modular. However, the matrix $\polyM$ satisfies $\det\polyM = 1 \notin \polyS$. Nevertheless, if we evaluate at some $\intvar\in\Z\backslash\lbrace -2,-1,0,1\rbrace$, we avoid this issue for this particular matrix since $\det\bM(\intvar) = 1\notin S(\intvar)$. One of our main results in this section states that this is a general phenomenon. 
	Let 
	\begin{align*}
		I(\polyS) = \lbrace \intvar\in\Z : s(\intvar) = f(\intvar)\text{ for some }s\in \polyS \text{ and }f\in F(\polyS)\rbrace
	\end{align*}
	denote the set of integer valued intersections between the polynomial functions given by the elements in $\polyS$ and $F(\polyS)$, the set of all polynomials that arise as a determinant of some forbidden minor for $\polyS$.
	
	\begin{lemma}\label{lemma_existence_equal_sets}
		Let $\polyS\subseteq \polyring$ be finite and $\intvar\in\Z\backslash I(\polyS)$. Then $\bM(\intvar)$ is totally $S(\intvar)$-modular if and only if $\polyM$ is totally $\polyS$-modular over $\polyring$. Also, the matrix $\bM(\intvar)$ is a forbidden minor for $S(\intvar)$ if and only if $\polyM$ is a forbidden minor for $\polyS$. Furthermore, the set $I(\polyS)$ is finite.
	\end{lemma}
	\begin{proof}
		Note that the statement about the forbidden minors follows directly from the first statement about totally $\polyS$-modular matrices. Therefore, we only need to prove the first part of the statement. 
		Recall that every totally $\polyS$-modular matrix over $\polyring$ is totally $S(\intvar)$-modular for every evaluation at $\intvar\in\Z$. So it suffices to show the other direction. 
		
		A totally $S(\intvar)$-modular matrix that is not the evaluation of a totally $\polyS$-modular matrix over $\polyring$ contains a forbidden minor $\polyM$ for $\polyS$ by definition. This forbidden minor $\polyM$ satisfies $\det\bM(\intvar)\in S(\intvar)$ and $\det\polyM\notin \polyS$. In other words, let $f\in F(\polyS)$ be the polynomial corresponding to $\det\polyM$, then $f\notin \polyS$ but $f(\intvar) = s(\intvar)$ for some $s\in \polyS$. Thus, we obtain $\intvar\in I(\polyS)$, a contradiction. So $\polyM$ does not contain a forbidden minor for $\polyS$ and is therefore totally $\polyS$-modular. Since $\polyS$ and $F(\polyS)$ are finite, see Lemma \ref{lemma_finite_number_of_determinants}, we get that there are only finitely many of those intersections. \qed
	\end{proof}
	
	We aim to make statements about matrices with entries being polynomials of degree at most one, that is, matrices given by $\polyM = \bT + \indet\cdot \bR\in\polyring^{n\times n}$ for integral $\bT$ and $\bR$. The following result relates the rank of $\bR$ to the largest degree among the polynomials in a given set $\polyS$. We write $\deg(s)$ for the degree of $s\in \polyring$.

	\begin{lemma}
		\label{lemma_rank_coefficient_matrix}
		Let $S\subseteq \polyring$ be finite. Let $\polyM = \bT + \indet\cdot \bR \in\polyring^{m\times n}$ be totally $\polyS$-modular for integral $\bT$ and $\bR$. Then we have $\rank\bR\leq \max_{s\in \polyS}\deg(s)$. 
	\end{lemma}
	\begin{proof}
		Let $r = \rank \bR$ and $\bR'\in\Z^{r\times r}$ be an invertible submatrix of $\bR$ and $\bM' = \bT' + \indet \cdot \bR'$ the corresponding submatrix of $\bM$.  
		We get
		\begin{align*}
			\det\bM' = \det\left(\bT' + \indet \cdot \bR'\right) = \det\bR'\cdot\det\left(\bR'^{-1}\bT' + \indet \cdot \bm{I}\right),
		\end{align*}
		where $\bm{I}$ denotes the unit matrix. Using the Leibniz formula for the determinant $\det(\bR'^{-1}\bT' + \indet \cdot \bm{I})$, we observe that there exists only one term with degree $r$, namely the term corresponding to the identity permutation, and no term with degree larger than $r$. Thus, the right hand side in the equation above is a polynomial of degree $r$. So $\det\polyM'$ is also a polynomial of degree $r$. Since the matrix $\polyM'$ is totally $\polyS$-modular, we have that $r = \deg(\det\polyM')\leq \max_{s\in \polyS}\deg(s)$.\qed
	\end{proof}
	
	So, if $\max_{s\in \polyS}\deg(s) = 1$ and $\bR$ is non-zero, then Lemma~\ref{lemma_rank_coefficient_matrix} implies that $\rank\bR = 1$. Hence, we can express $\bR$ as a rank-1 update and obtain $\bM = \bT + \indet\cdot \bu\cdot \bv^\top$ for some integral $\bu$ and $\bv$. 
	Given such a matrix $\bM = \bT + \indet\cdot \bu\cdot \bv^\top$ without any restriction on its subdeterminants, it is possible to establish that subdeterminants of $\polyM$ are indeed polynomials of degree at most one. This is a consequence of the well-known matrix determinant lemma for rank-1 updates. A variant of this lemma is stated below. 
	\begin{lemma}[Matrix determinant lemma]
		\label{lemma_matrix_determinant_lemma}
		Let $\polyM = \bT+ \indet \cdot\bu\cdot\bv^\top\in\polyring^{n\times n}$ for integral $\bT$, $\bu$, and $\bv$. Then we have
		\begin{align*}
			\det \polyM = \left(\det\bT - \det\left(\bT- \bu\cdot\bv^\top\right)\right)\cdot \indet + \det\bT.
		\end{align*}
	\end{lemma}
	This is the essential ingredient to prove Theorem \ref{thm_linear_form_decomposition}. 
	
	\begin{proof}[of Theorem \ref{thm_linear_form_decomposition}]
		Let $\polyM = \bT + \indet \cdot\bu\cdot\bv^\top$ be totally $\polyS$-modular. Then Lemma \ref{lemma_matrix_determinant_lemma} implies that $\det\polyM = \lambda_1\indet+\lambda_0$ for $\lambda_0,\lambda_1\in\Z$ and $\det\bT = \lambda_0$. This holds for all such totally $\polyS$-modular matrices. So the claim follows.\qed
	\end{proof}
	We will also use an implication of Theorem \ref{thm_linear_form_decomposition} for matrices over $\Z$. 
	\begin{corollary}\label{cor_decomposition_evaluation}
		Let $\polyS\subseteq\polyring$ be finite and $\intvar\in \Z\backslash I(\polyS)$. If $\bM(\intvar) = \bT + \intvar\cdot\bu\cdot\bv^\top$ with integral $\bT$, $\bu$, and $\bv$ is totally $S(\intvar)$-modular, then the matrix $\bT$ is totally $\polyS(0)$-modular.
	\end{corollary}
	\begin{proof}
		Since $\intvar\in \Z\backslash I(\polyS)$, we know that every totally $S(\intvar)$-modular matrix $\bM(\intvar) = \bT + \intvar\cdot\bu\cdot\bv^\top$ corresponds to a totally $\polyS$-modular matrix $\polyM=\bT + \indet \cdot\bu\cdot\bv^\top$ over $\polyring$ by Lemma \ref{lemma_existence_equal_sets}. From Theorem \ref{thm_linear_form_decomposition} it follows that $\bT$ is totally $\polyS(0)$-modular.\qed
	\end{proof}

	\section{Properties of $I(\polyS)$}
	Following Lemma \ref{lemma_existence_equal_sets}, we demonstrate how to calculate all intersections between the polynomial functions in a given set $\polyS$ and the polynomial functions in $F(\polyS)$. 
	We start by proving a lemma which holds under more general assumptions. 
	For that purpose, we remark that the ring $\polyring$ is a unique factorization domain. So every element in $\polyring$ admits a unique factorization into smaller irreducible elements up to multiplying with $\pm 1$. Therefore, the notion of greatest common divisors carries naturally over from $\Z$ to $\polyring$. Analogously to $\Z$, elements in $\polyring$ are relatively prime if their greatest common divisor is $1$. We call a matrix $\polyM\in\polyring^{m\times n}$ \emph{totally unimodular} if every subdeterminant is contained in $\pm\lbrace 0, 1\rbrace$.
	\begin{lemma}\label{lemma_to_diff_determinants}
		Let $\polyS \subseteq\polyring$ be such that all non-zero $y,z\in \polyS$ are relatively prime or $\left|y\right| = \left|z\right|$ and $2\notin \polyS$. Let $n\geq3$ and $\polyM\in\polyring^{n\times n}$ be invertible and every $(n-1)\times (n-1)$ submatrix of $\polyM$ is totally $\polyS$-modular. Then either $\polyM$ is totally unimodular or there exist invertible submatrices $\bA_{(n-2)},\tilde{\bA}_{(n-2)}\subseteq\bM$ such that
		\begin{align*}
			\left|\det\bA_{(n-2)}\right|\neq\bigl|\det\tilde{\bA}_{(n-2)}\bigr|. 
		\end{align*}
	\end{lemma}
	\begin{proof}
		Let $n=3$. The $(n-2)\times (n-2)$ submatrices of $\bM$ are the entries of $\bM$. Assume that every entry of $\bM$ is in $\pm\lbrace 0, s\rbrace$ for some $s\in S$. If $s=1$, we get that $\bM\in\pm\lbrace 0,1\rbrace^{3\times 3}$. Since $2\notin S$, the matrix $\bM$ is totally unimodular, cf. \cite[Theorem 19.3]{schrijvertheorylinint86}. If we assume $s\neq 1$, we have  $\det\bA_{(2)}\in\pm\lbrace s^2, 2s^2\rbrace$ for an invertible submatrix $\bA_{(2)}\subseteq\bM$ as every entry is contained in $\pm\lbrace 0, s\rbrace$. Since the non-zero elements in $S$ are pairwise relatively prime and $s\in S$, we get $\det\bA_{(2)}\notin S$, contradicting that $\bA_{(2)}$ is totally $S$-modular. 
		
		So we suppose that $n\geq 4$. Let $\bM$ not be totally unimodular. Again, we assume that every invertible $(n-2)\times (n-2)$ submatrix has determinant $s$ or $-s$ for some $s\in S$. By the induction hypothesis applied to some invertible submatrix $\bA_{(n-1)}\subseteq\bM$, there exists an invertible submatrix $\bA_{(n-3)}\subseteq\bA_{(n-1)}$ with $\bigl|\det\bA_{(n-3)}\bigr|\neq  \left|s\right|$. 
		We apply the Desnanont-Jacobi identity, (\ref{eq_desnanont_jacobi}), to $\bA_{(n-3)}\subseteq \bA_{(n-1)}$ and obtain that 
		\begin{align*}
			\det\bA_{(n-1)} = \frac{1}{\det\bA_{(n-3)}} \left( s_1s_2-s_3s_4\right)\in\pm\frac{1}{\det\bA_{(n-3)}}\cdot\lbrace s^2, 2s^2\rbrace
		\end{align*}
		with $s_i\in\pm\lbrace 0,s\rbrace$ for $i=1,2,3,4$. If $\left|\det\bA_{(n-3)}\right| \neq 1$, then $\det\bA_{(n-3)}$ does not divide $s^2$ or $2s^2$ as the non-zero elements in $S$ are pairwise relatively prime and $2\notin S$. However, this implies that $\det\bA_{(n-1)}\notin \polyring$, a contradiction. If $\left|\det\bA_{(n-3)}\right| =1$, we get that $\det\bA_{(n-1)}\in \pm\lbrace s^2,2s^2\rbrace$. This gives us a contradiction since $\det\bA_{(n-1)}$ and $s$ are not relatively prime. \qed
		
	\end{proof}
	
	We are in the position to determine the set of possible determinants given by the forbidden minors for a specific set $\polyS$. We showcase this for $\polyS=\pm\lbrace 0,1,\indet,\indet+1,2\indet+1\rbrace$.
	\begin{lemma}\label{lemma_determinants_minors_1_a_a+1_2a+1}
		Let $\polyS=\pm\lbrace 0,1,\indet,\indet+1,2\indet+1\rbrace$ and $D$ be the set of all determinants attained by a $2\times 2$ forbidden minor for $\polyS$. Then we have
		\begin{align*}
			F(\polyS)\subseteq \pm\lbrace 2,\indet-1,\indet+2,2\indet,2\indet+ 2,3\indet+1,3\indet+2,4\indet+2\rbrace\cup D.
		\end{align*}
	\end{lemma}
	\begin{proof}
		Let $n\geq 3$ and $\polyM$ be a forbidden minor. There exists an invertible submatrix $\bA_{(n-2)}\subseteq\polyM$. Applying the Desnanont-Jacobi identity (\ref{eq_desnanont_jacobi}), we get
		\begin{align*}
			\det\polyM\cdot\det\bA_{(n-2)} = s_1s_2 - s_3s_4
		\end{align*}
		with $s_i\in \polyS$ for $i=1,2,3,4$. This equation needs to have a solution for $\det\polyM$ in $\polyring$. Suppose that $\det\polyM$ has degree larger than one. This implies $\left|\det\bA_{(n-2)}\right| = 1$, an equality that has to hold for every invertible $(n-2)\times (n-2)$ subdeterminant. By Lemma \ref{lemma_to_diff_determinants}, we deduce that $\polyM$ is totally unimodular and, thus, not a forbidden minor. So we assume $\det\bM$ has degree at most one. Using suitable software such as SageMath, one can enumerate all feasible solutions to the Desnanont-Jacobi identity with degree at most one. This gives us the following feasible values for $\det\polyM$ up to a sign:
		\begin{align*}
			2\polyS\backslash 0 \cup \lbrace \indet-1,\indet+2,2\indet-1,2\indet+3,3\indet+1,3\indet+2,4\indet,4\indet+1,4\indet+3,4\indet+4\rbrace.
		\end{align*}
		However, it is possible to show that the values $\pm\lbrace 2\indet-1,2\indet+3,4\indet,4\indet+1,4\indet+3,4\indet+4\rbrace$ for $\det\polyM$ can only appear for a unique choice of $\left|\det\bA_{(n-2)}\right|$. Therefore, if there exists a forbidden minor with such a determinant, every invertible $(n-2)\times (n-2)$ submatrix needs to have the same determinant in absolute value. This contradicts Lemma \ref{lemma_to_diff_determinants}. Hence, we can exclude these values and obtain the determinants from the statement.\qed
	\end{proof}
	Computing the intersections of the elements in $\polyS=\pm\lbrace 0,1,\indet,\indet+1,2\indet+1\rbrace$ and the set in Lemma \ref{lemma_determinants_minors_1_a_a+1_2a+1} yields that $I(\polyS)\subseteq\lbrace -3,\ldots,2\rbrace$ in Corollary \ref{cor_decomposition_evaluation}. We remark that it is possible to show that both sets are indeed equal.
	
	\begin{remark}
		\label{rem_intersections_1_x_x+1_2x+1}
		For the set $\polyS = \pm \lbrace 0,\indet,\indet + 1,2\indet + 1\rbrace$, one can show analogously to the proof of Lemma~\ref{lemma_determinants_minors_1_a_a+1_2a+1} that $F(S) \subseteq \pm\lbrace 1,2\indet, 2\indet + 2, 3\indet + 1, 3\indet + 2, 4\indet + 2\rbrace \cup D$ and $I(S) = \lbrace -2,-1,0,1\rbrace$, where $D$ is the set of all determinants attained by a $2\times 2$ forbidden minor for $\polyS$.
	\end{remark}

	\section{Proofs of Theorems \ref{thm_recognition} and \ref{thm_opt_a_a_a+1}}
	For both cases, $S=\pm\lbrace 0,1,\indet,\indet+1,2\indet+1\rbrace$ and $S=\pm\lbrace 0,\indet,\indet+1,2\indet+1\rbrace$, we recover that $S(\intvar)=\pm\lbrace 0,1\rbrace$ if $\intvar \in \lbrace -1,0\rbrace$. 
	Since one can optimize $\ILP(\bM(\intvar),\bb,\bc)$ in polynomial time for totally unimodular constraint matrices, see, for instance, \cite[Chapter 19]{schrijvertheorylinint86}, and recognize totally unimodular matrices in polynomial time \cite[Chapter 20]{schrijvertheorylinint86}, both theorems hold in this case. In the following, we prove Theorems \ref{thm_recognition} and \ref{thm_opt_a_a_a+1} for the other values of $\intvar$.
	
	\subsection{Proof of Theorem \ref{thm_recognition}}\label{ss_recognition_proof}
	Recall that $S=\pm\lbrace 0,1,\indet,\indet+1,2\indet+1\rbrace$. By Lemma~\ref{lemma_rank_coefficient_matrix}, every totally $\polyS$-modular matrix $\polyM$ admits a decomposition into $\polyM = \bT + \indet\cdot\bm{u}\cdot\bm{v}^\top$ for some suitable integral valued $\bT$, $\bu$, and $\bv$.
	One can test in polynomial time whether such a decomposition exists. 
	To finish the proof of Theorem~\ref{thm_recognition}, we use the complete characterization given below and the fact that one can recognize totally unimodular matrices in polynomial time; see \cite[Chapter 20]{schrijvertheorylinint86}. 
	\begin{lemma}
		\label{lemma_t_bar_t_tu}
		Let $\bT$, $\bu$,  and $\bv$ be integral. The matrix $\polyM = \bT + \indet \cdot\bm{u}\cdot\bm{v}^\top$ is totally $\polyS$-modular if and only if $\bT$ and $\bT - \bu\cdot\bv^\top$ are totally unimodular.
	\end{lemma}
	\begin{proof}
		We abbreviate $\bar{\bT} = \bT - \bu\cdot\bv^\top$. It suffices to show the statement for square matrices. So we assume without loss of generality that $m = n$. 
		Let $\bT$ and $\bar{\bT}$ be totally unimodular. We have
		\begin{align}\label{eq_proof_recognition_matrix_determinant_lemma}
			\det\polyM = \left(\det\bT - \det\bar{\bT}\right)\cdot \indet + \det\bT
		\end{align}
		by Lemma \ref{lemma_matrix_determinant_lemma}. So $\det\polyM$ is completely determined by $\det\bT$ and $\det\bar{\bT}$. 
		As $\det\bT$ and $\det\bar{\bT}$ are both in $\pm\lbrace 0, 1\rbrace$, we obtain that $\det\polyM\in \pm\lbrace 0, 1, \indet, \indet + 1, 2\indet + 1\rbrace = \polyS$ by going through all feasible cases. The other direction follows directly by evaluating $\polyS$, $\polyM$, and (\ref{eq_proof_recognition_matrix_determinant_lemma}) at $\indet\in\lbrace -1,0\rbrace$.\qed
	\end{proof}
	By Lemma \ref{lemma_existence_equal_sets} and $I(\polyS)\subseteq\lbrace -3,\ldots,2\rbrace$ from the previous section, we can now derive Theorem~\ref{thm_recognition} by replacing $S$ with $S(\intvar)$ for  $\intvar\in\Z\backslash\lbrace-3,-2,-1,0,1,2\rbrace$.
	
	\subsection{Proof of Theorem \ref{thm_opt_a_a_a+1}}\label{ss_optimization_proof}
	Following Remark~\ref{rem_intersections_1_x_x+1_2x+1}, we fix some $\intvar\in\Z\backslash\lbrace -2,-1,0,1\rbrace$. We use the decomposition $\bM(\intvar) = \bT + \intvar\cdot\bm{u}\cdot\bm{v}^\top$ for integral $\bu,\bv$ from Lemma~\ref{lemma_rank_coefficient_matrix} where $\bT$ is totally unimodular by Corollary~\ref{cor_decomposition_evaluation}. By multiplying rows and columns with $-1$, we assume without loss of generality that $\bu$ and $\bv$ are non-negative. Observe that this already implies that the entries of $\bT$ are in $\lbrace 0,1\rbrace$ since the entries of $\bM(\intvar)$ are in $S(\intvar)$. 
	Let without loss of generality $\bM(\intvar)$ have no zero column or row. This implies that no entry of $\bu$ and $\bv$ equals zero. 
	If there exists an entry of $\bM(a)$ which is $2\intvar + 1$, then due to the rank-1 update the whole row and/or column containing this entry have entries $2\intvar + 1$. 
	If there exists a row of $\bM(a)$ whose entries are all $2\intvar + 1$, we can divide this row by $\intvar/(2\intvar + 1)$ and round down the right hand side. So we can always assume that $\bu = \bm{1}$. 
	If there are multiple columns whose entries are all $2\intvar + 1$, we aggregate them. Hence, we assume that there is at most one column with entries $2\intvar + 1$. So we can select $\bv\in \lbrace 1,2\rbrace^{n}$ where at most one entry of $\bv$ is $2$. 
	We suppose without loss of generality that $\bv_1\in\lbrace 1,2\rbrace$ and $\bv_{\backslash 1} = \bm{1}$. 
	Let the first column of $\bM(a)$ be $\bM(a)_{\cdot,1}$. Notice that $\bM(a)_{\cdot,\backslash 1}$ is a matrix with entries in $\lbrace \intvar,\intvar + 1\rbrace$. We set $y = \bm{1}^\top\tilde{\bx}$ for $\bx = (\bx_1, \tilde{\bx})^\top$ to reformulate $\ILP(\bM(\intvar),\bb,\bc)$ in the form of the mixed integer linear program
	\begin{align*}
		\max \bc^\top\bx \ \text{ s.t. } \ \bT_{\backslash 1,\cdot}\tilde{\bx} \leq\bb - \bx_1 \cdot \bM(a)_{\cdot,1} - \intvar\cdot\bu\cdot y, \ \bm{1}^\top\tilde{\bx} = y, \ \bx_1,y\in\Z.
	\end{align*}
	We next show that the constraint matrix, the transpose of $[\bT_{\backslash 1,\cdot}| \bm{1}]$, is totally unimodular. If this is true, then a solution of the mixed integer linear program in two integer variables can be obtained in polynomial time using Lenstra's algorithm \cite{lenstraintprogr83} or its improved successors \cite{dadush2012latticealgo,kannan1987integeropt,reisroth2023flatness}. It corresponds to a solution of $\ILP(\bM(\intvar),\bb,\bc)$. 
	
	Consider $[\bT_{\backslash 1,\cdot}| \bm{1}]$. The matrix $\bT_{\backslash 1, \cdot}$ is totally unimodular since its a submatrix of $\bT$. Let $\bT'$ be an invertible submatrix of the constraint matrix that contains the row $\bm{1}$. 
	By adding $\intvar\cdot\bm{1}$ to every row of $\bT'$ that is not the all-ones row, we obtain a new matrix $\bM(\intvar)'$ that contains a submatrix of $\bM(a)_{\cdot, \backslash 1}$, whose entries are in $\lbrace \intvar,\intvar + 1\rbrace$, and one all-ones row. Since each submatrix of $\bM(\intvar)$ is totally $S(\intvar)$-modular, we obtain $\pm 1= \det\bM(\intvar)' = \det\bT'$ by Theorem~\ref{thm_append_all_one} below. 
	So $[\bT_{\backslash 1,\cdot}| \bm{1}]$ is totally unimodular. 
	\begin{theorem}\label{thm_append_all_one}
		Let $S = \pm \lbrace 0,\indet,\indet + 1, 2\indet + 1\rbrace\subseteq\polyring$ and $n\geq 2$. Let $\polyM\in\lbrace \indet,\indet+1\rbrace^{n\times (n-1)}$ be totally $\polyS$-modular. Then $\det[\polyM|\bm{1}]\in\pm\lbrace 0,1\rbrace$.
	\end{theorem}
	\begin{proof}
		For the purpose of deriving a contradiction, assume that the statement does not hold. We multiply the all-ones column with $\indet + 1$. So we deduce that $\det[\polyM|(\indet+1)\cdot\bm{1}]\notin \pm\lbrace 0,\indet+1\rbrace$ by assumption. 
		Note that $\det[\polyM|(\indet+1)\cdot\bm{1}]$ is divisible by $\indet+1$. So we obtain that $\det[\polyM|(\indet+1)\cdot\bm{1}]\notin \polyS$. Thus, the matrix $[\polyM|(\indet+1)\cdot\bm{1}]$ contains a forbidden minor for $\polyS$. As $\polyM$ is totally $\polyS$-modular, the forbidden minor contains the all-$(\indet+1)$'s column. For the remainder of the proof, we show that no forbidden minor with entries in $\lbrace\indet,\indet + 1\rbrace$ contains an all-$(\indet + 1)$ column or row, which finishes the proof of the theorem.
		
		One can verify that for $n=2$ no forbidden minor with entries in $\lbrace\indet,\indet + 1\rbrace$ exists. So we assume that $n\geq 3$. 
		Suppose without loss of generality that the last row of $\bM$, which we denote by $\bM_{n,\cdot}$, has greatest common divisor larger than $1$. We derive a contradiction. 
		
		We observe that $\gcd\bM_{n,\cdot}\in \lbrace \indet,\indet + 1\rbrace$. By Remark~\ref{rem_intersections_1_x_x+1_2x+1}, we have $\left|\det\bM\right| = 2\cdot\gcd\bM_{n,\cdot}$ as $\det\bM$ has to be divisible by $\gcd\bM_{n,\cdot}$. 
		Let $\bA_{(n-1)}\subseteq\bM$ be an invertible submatrix not containing the last row. 
		Take an invertible submatrix $\bA_{(n-2)}\subseteq\bA_{(n-1)}$. The Desnanont-Jacobi identity, (\ref{eq_desnanont_jacobi}), applied to $\bA_{(n-2)}$ yields
		\begin{align*}
			2\gcd\bM_{n,\cdot}\det\bA_{(n-2)} = \det\bM\det\bA_{(n-2)} = s_1s_4 - s_2s_3,
		\end{align*}
		where $s_1,s_2\in \polyS$ and $s_3,s_4\in\pm \lbrace 0, \gcd\bM_{n,\cdot}\rbrace$. 
		This equation cannot be satisfied if $s_i = 0$ for some $i=1,2,3,4$. So we assume that $s_i\neq 0$ for all $i=1,2,3,4$. By division with $\gcd\bM_{n,\cdot}$ and assuming that without loss of generality $s_3=\gcd\bM_{n,\cdot}=s_4$, we get $2\cdot\det\bA_{(n-2)} = s_1 - s_2$ for $s_1,s_2\in \polyS\backslash 0$. The only solutions are given when $\bigl|\det\bA_{(n-2)}\bigr|=s_1= - s_2$ or $\bigl|\det\bA_{(n-2)}\bigr|= - s_1= s_2$. 
		Hence, every invertible $(n-2)\times (n-2)$ submatrix of $\bA_{(n-1)}$ has the same determinant as $\bA_{(n-1)}$ in absolute value. 
		If $n=3$, this implies that the $(n-2)\times (n-2)$ subdeterminants correspond to entries. As the entries of $\bM$ are in $\lbrace \indet,\indet+1\rbrace$, this means that all the entries in $\bA_{(n-1)}$ are the same which gives us $\det\bA_{(n-1)} = 0$, a contradiction. So let $n\geq 4$. Take an invertible submatrix $\bA_{(n-3)}\subseteq\bA_{(n-1)}$. We apply the Desnanont-Jacobi identity, (\ref{eq_desnanont_jacobi}), to $\bA_{(n-3)}\subseteq\bA_{(n-1)}$ and get $\det\bA_{(n-3)}\in\pm\lbrace \det\bA_{(n-1)},2\det\bA_{(n-1)}\rbrace$ which yields $\left|\det\bA_{(n-3)}\right| = \left|\det\bA_{(n-1)}\right|$. This holds for all invertible $(n-3)\times (n-3)$ submatrices of $\bA_{(n-1)}$. By Lemma \ref{lemma_to_diff_determinants}, this implies that $\bA_{(n-1)}$ is totally unimodular. However, this contradicts that $\det\bA_{(n-1)}\in \polyS$.\qed		
	\end{proof}	
	
	\section{Finiteness of forbidden minors}\label{sec_finite_fm}
	Recall from the introduction that totally $\pm\lbrace 0,\indet,\indet+1\rbrace$-modular matrices with entries in $\lbrace \indet,\indet+1\rbrace$ are completely characterized by excluding the submatrix (\ref{matrix_forbidden_conflict}). 
	In light of the celebrated Robertson-Seymour theorem \cite{ROBERTSON2004325}, one might ask whether there is always a finite list of forbidden minors for a finite set $\polyS\subseteq\polyring$ and matrices with entries in $\lbrace \indet,\indet+1\rbrace$. This is not the case. 
	An intriguing example with infinitely many forbidden minors is already given by the set $\polyS=\pm\lbrace 0,1,\indet,\indet+1,2\indet+1\rbrace$; see Figure \ref{fig_forbidden minors} for an incomplete list. Interestingly, if one removes the $1$ from $\polyS$ and passes to $\pm\lbrace 0,\indet,\indet+1,2\indet+1\rbrace$, then we are only aware of finitely many forbidden minors; see Figure \ref{fig_forbidden minors}. It is open whether this list is complete. If this is true, this might support the following conjecture for a finite set $\polyS\subseteq\polyring$ such that $s\in S$ implies $-s\in S$: there exists a finite list of forbidden minors if and only if $1\notin \polyS$. 
	
	\begin{figure}
		\includegraphics[scale=.75]{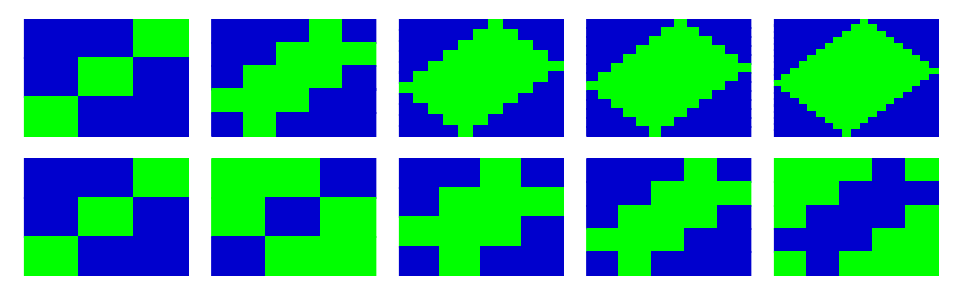}
		\caption{The blue boxes depict the value $\indet$ and the green boxes $\indet+1$ or vice versa. The first row of matrices corresponds to the first five elements of an infinite sequence of matrices that can be obtained by generalizing the existing pattern. It can be shown that those infinitely many matrices are forbidden minors for $\pm\lbrace 0,1,\indet,\indet+1,2\indet+1\rbrace$. The matrices in the last row correspond to five forbidden minors for $\pm \lbrace 0,\indet,\indet+1,2\indet+1\rbrace$.}
		\label{fig_forbidden minors}
	\end{figure}
	
	\subsubsection{Acknowledgements.} The authors are grateful to the reviewers for their valuable suggestions and comments.

	\bibliographystyle{plain}
	\bibliography{references}
\end{document}